\documentclass[a4paper,12pt]{amsart}
\usepackage[latin1]{inputenc}
\usepackage{amssymb}
\usepackage{mathrsfs}
\usepackage{enumerate}
\usepackage{vmargin}
\usepackage{verbatim}

\usepackage{ifpdf}
\ifpdf 
\usepackage[pdftex]{hyperref} 
\else
\usepackage[hypertex]{hyperref}
\fi

\hypersetup{
citecolor=green,
menucolor=red,
citebordercolor=0 1 0,
linkbordercolor=1 0 0,
pdfpagemode=UseOutlines,
pdftitle={Instability for standing waves of nonlinear Klein-Gordon equations via mountain-pass arguments},
pdfauthor={Stefan LE COZ <slecoz@univ-fcomte.fr>},
pdfsubject={Instability for standing waves of nonlinear Klein-Gordon equations via mountain-pass arguments},
pdfkeywords={Instability, standing waves, nonlinear Klein-Gordon equations, mountain-pass arguments, blow-up, variational methods}
}

\setmarginsrb {3.5cm}{3cm}{3.5cm}{4cm} {1cm}{1cm}{1cm}{1cm}
\newtheorem{defi}{Definition}
\newtheorem{thm}{Theorem}
\newtheorem{prop}[defi]{Proposition}
\newtheorem{remq}[defi]{Remark}
   \newenvironment{rmq}{\begin{remq}\rm}{\end{remq}}
\newtheorem{lem}[defi]{Lemma}

\newcommand{\nlpsps}[1]{\|#1\|_{p+1}^{p+1}}
\newcommand{\nld}[1]{\|#1\|_2}
\newcommand{\nldd}[1]{\|#1\|_2^2}
\newcommand{\nhu}[1]{\|#1\|_{\hu}}
\newcommand{\nhud}[1]{\|#1\|_{\hu}^2}
\newcommand{\nhurd}[1]{\|#1\|_{\hurd}}

\newcommand{\dual}[2]{\left< #1,#2 \right>}

\renewcommand{\a}{\alpha}
\renewcommand{\b}{\beta}
\renewcommand{\d}{\delta}
\newcommand{\e}{\varepsilon}
\newcommand{\f}{\varphi}
\newcommand{\g}{\gamma}
\renewcommand{\l}{\lambda}
\newcommand{\m}{\mu}
\newcommand{\n}{\nu}
\renewcommand{\r}{\rho}
\renewcommand{\t}{\theta}
\newcommand{\w}{\omega}

\newcommand{\D}{\Delta}
\newcommand{\G}{\Gamma}
\renewcommand{\L}{\Lambda}

\newcommand{\C}{\mathbb{C}}
\newcommand{\N}{\mathbb{N}}
\newcommand{\R}{\mathbb{R}}
\newcommand{\Rd}{\mathbb{R}^{2}}
\newcommand{\RN}{\mathbb{R}^N}

\newcommand{\ld}{L^2(\RN)}

\newcommand{\ldradrd}{L^2_{{\rm{rad}}}(\Rd)}
\newcommand{\ldrad}{L^2_{{\rm{rad}}}(\RN)}

\renewcommand{\lq}{L^q(\RN)}
\newcommand{\hu}{H^1(\RN)}
\newcommand{\hurad}{H^1_{{\rm{rad}}}(\RN)}
\newcommand{\huradrd}{H^{1}_{{\rm{rad}}}(\Rd)}
\newcommand{\hurd}{H^{1}(\Rd)}

\newcommand{\Kab}{K_{\a,\b}}

\newcommand{\uo}{u_0}
\newcommand{\ue}{u_{\e}}

\newcommand{\ut}{u_{t}}
\newcommand{\utt}{u_{tt}}
\newcommand{\utb}{\overline{\ut}}

\newcommand{\vo}{v_{0}}
\newcommand{\vl}{v_{\l}}
\newcommand{\vlo}{v_{\lo}}

\newcommand{\fw}{\f_{\w}}
\newcommand{\fl}{\f_{\l}}
\newcommand{\flm}{\f_{\l,\m}}
\newcommand{\fe}{\f_{\e}}

\newcommand{\tz}{t_{0}}

\newcommand{\la}{\l^{\a}}
\newcommand{\lb}{\l^{\b}}
\newcommand{\lo}{\l_{0}}
\newcommand{\lm}{\l_{\m}}

\newcommand{\Lm}{\L_{\m}}

\newcommand{\Ca}{C_{\a}}

\newcommand{\scrKab}{\mathscr{K}_{\a,\b}}
\newcommand{\scrP}{\mathscr{P}}
\newcommand{\I}{\mathcal{I}}

\renewcommand{\Re}{\mathrm{Re}}

\renewcommand{\leq}{\leqslant}
\renewcommand{\geq}{\geqslant}

\newcommand{\goestoweak}{\rightharpoonup}

\newcommand{\intrn}{\int_{\RN}}
\newcommand{\intrd}{\int_{\Rd}}

\begin{document}


\title[Instability via mountain-pass arguments]
{Instability for standing waves of nonlinear Klein-Gordon equations via mountain-pass arguments}

\date\today

\author{Louis JEANJEAN and Stefan LE COZ}
\address[Louis JEANJEAN \and Stefan LE COZ]{ Laboratoire de Math\'{e}matiques,
\endgraf
Universit\'{e} de Franche-Comt\'{e},\endgraf 25030 Besan\c{c}on CEDEX,\endgraf FRANCE.}
\email{louis.jeanjean@univ-fcomte.fr} \email{slecoz@univ-fcomte.fr}

\begin{abstract}
We introduce mountain-pass type arguments in the context of orbital instability
for Klein-Gordon equations. Our aim is to illustrate on two examples how these
arguments can be useful to simplify proofs and derive new results of orbital
stability/instability. For a power-type nonlinearity, we prove that the ground
states of the associated stationary equation are minimizers of the functional
\emph{action} on a wide variety of constraints. For a general nonlinearity, we
extend to the dimension $N=2$ the classical instability result for stationary
solutions of nonlinear Klein-Gordon equations proved in 1985 by Shatah in
dimension $N\geq3$.
\end{abstract}

\maketitle

\section{Introduction}

The aim of the present paper is to show how recent methods and results
concerning the variational characterizations of the ground states for elliptic
equations of the form
\begin{equation}\label{eq:nlsf1}
-\D \f = g(\f),  \quad \f \in H^1(\R^N; \C)
\end{equation}
can be used to study the orbital stability/instability of the standing waves of
various nonlinear equations such as Schr\"{o}dinger equations, Klein-Gordon
equations, generalized Boussinesq equations, etc. Our work is motivated by
recent developments (see for instance \cite{cg,li,lot,lww,ot1,ot2}) of the
techniques introduced by Berestycki and Cazenave \cite{bc} to prove the
instability of standing waves for nonlinear evolution equations. We present our
approach on two examples involving nonlinear Klein-Gordon equations of the form
\begin{equation}\label{nlkgf}
\utt-\D u+\r u=f(u)
\end{equation}
where $\r>0$, $u:\R\times\RN\mapsto\C$ and $f:[0,+\infty)\mapsto\R$ is extended to $\C$ by setting $f(z)=f(|z|)z/|z|$ for $z\in\C\setminus\{0\}$.

A standing wave of (\ref{nlkgf}) is a solution of the form $e^{i\w t}\fw(x)$
for $\w\in\R$ and $\fw\in H^1(\R^N; \C)$. Thus $\fw$ satisfies
\begin{equation}\label{eq:snlkgf}
-\D \fw +(\r-\w^2)\fw- f(\fw)=0.
\end{equation}
Clearly, (\ref{eq:snlkgf}) is of the form (\ref{eq:nlsf1}). From now on we
write $\hu$ for $H^1(\R^N; \C)$. The \emph{least energy level} $m$ is defined
by
\begin{equation}\label{eq:m}
m:=\inf\{S(v)\big|v\in\hu\setminus\{0\},\,v\mbox{ is a solution of (\ref{eq:nlsf1})} \}
\end{equation}
where $S:\hu\mapsto\R$ is the natural functional (often called \emph{action}) corresponding to (\ref{eq:nlsf1})
$$
S(v):=\frac{1}{2}\nldd{\nabla v}-\intrn G(v)dx,
$$
with $G(s):=\int_0^{|s|} g(t)dt$. A solution $\f\in\hu$ of (\ref{eq:nlsf1}) is said to be a \emph{ground state}, or \emph{least energy solution}, if
$$
S(\f)=m.
$$
The study of the existence for solutions of (\ref{eq:nlsf1}) goes back to the
work of Strauss \cite{st} (see also \cite{cgm}). The most general result in
that direction is due to Berestycki and Lions \cite{bl} for $N=1$ and $N\geq3$
and Berestycki, Gallouet and Kavian \cite{bgk} for $N=2$.

The assumptions of \cite{bgk,bl} when $N \geq 2$ are :
\begin{itemize}
\item[(g0)] $g$ is continuous and odd,
\item[(g1)] if $N\geq 3$, $\displaystyle-\infty<\liminf_{s\to 0}\frac{g(s)}{s}\leq\limsup_{s\to 0}\frac{g(s)}{s}<0,$\\
if $N=2$, $\displaystyle-\infty<\lim_{s\to 0}\frac{g(s)}{s}:=-\r<0$,
\item[(g2)] if $N\geq3$, $\displaystyle\lim_{s\to+\infty}\frac{g(s)}{s^{\frac{N+2}{N-2}}}=0$,\\
if $N=2$, $\forall \a>0$ $\exists \Ca>0$ such that $|g(s)|\leq \Ca e^{\a s^2}$ $\forall s>0$.
\item[(g3)] there exists $\xi_0>0$ such that $G(\xi_0)>0$.
\end{itemize}
It is known that the assumptions (g0)-(g3) are almost optimal to insure the
existence of a solution of  (\ref{eq:nlsf1}) (see \cite[Section 2.2]{bl}). In
\cite{bgk,bl} it is proved that for $N\geq2$ and under (g0)-(g3) there exists a
positive radial least energy solution $\f$ of (\ref{eq:nlsf1}) when the infimum
in (\ref{eq:m}) is taken over the solutions belonging to $H^1(\R^N, \R)$.
Moreover it is easily deduce from the proofs in \cite{bgk,bl} that this $\f$ is
still a least energy solution of (\ref{eq:nlsf1}) when the infimum is, as in
(\ref{eq:m}), taken over the set of all complex valued solutions. See
\cite{cjs} for a proof of this statement along with a description of the ground
states as being of the form $U = e^{i \theta} \tilde{U}$ where $\t\in\R$ and $\tilde{U}$ is a
real positive ground state solution of (\ref{eq:nlsf1}).

In dimension $N=1$, the assumptions in \cite{bl} are 
\begin{itemize}
\item[(h0)] $g$ is locally Lipschitz continuous and $g(0)=0$,
\item[(h1)] there exists $\eta_0>0$ such that
$$
G(s)<0\text{ for all }s\in(0,\eta_0),\;G(\eta_0)=0,\;g(\eta_0)>0
$$
\end{itemize}
and it is proved in \cite{bl} that under (h0) the condition (h1)
is necessary and sufficient to guarantee the existence of a unique (up to
translation) real positive solution of (\ref{eq:nlsf1}). Here also, it can be shown (see
\cite{cjs}) that the least energy levels coincide for complex and real valued
solutions of (\ref{eq:nlsf1}).

Since the pioneer works \cite{bc,cl}, it is known that the
stability/instability of the standing waves is closely linked to additional
variational characterizations that the associated ground states enjoy.
Recently, in \cite{jt1} for $N\geq2$ and in \cite{jt2} for $N=1$, Jeanjean and
Tanaka showed that, under the conditions (g0)-(g3) for $N\geq2$ and basically
(h0)-(h1) for $N=1$, the functional $S$ admits a mountain pass geometry.
Precisely they show that setting
\begin{equation}\label{Gamma}
\G:=\{ \g\in\mathcal{C}([0,1],\hu),\g(0)=0,\,S(\g(1))<0 \}
\end{equation}
one has $\G \neq \emptyset$ and
\begin{equation}\label{eq:c}
c:=\inf_{\g\in\G}\max_{t\in[0,1]}S(\g(t)) >0.
\end{equation}
Furthermore, they proved that 
$$
c=m,
$$
namely that the mountain pass value gives
the least energy level. In fact, the results of \cite{jt1,jt2} are proved within
the space $H^1(\R^N, \R)$ but it is straightforward to show, see Lemma
\ref{complex}, that this equality also holds in $\hu$.

In this paper, we will show, by studying two specific problems, how the ideas
and methods developed in \cite{jt1,jt2} can be implemented in the context of
instability by blow-up for nonlinear Klein-Gordon equations.

First, working with a nonlinearity of power type ($f(s)=|s|^{p-1}s$) we find a
set of constraints on which the ground states are minimizers of $S$. In
particular, this gives an alternative, much simpler proof of results in
\cite{lot,ot1,ot2} concerning the derivation of an additional variational
characterization of the ground states. Precisely, we prove
\begin{thm}\label{thm:snlkgp}
Let $\a,\b\in\R$ be such that
\begin{equation}\label{eq:rangeofab}
\left\{
\begin{array}{rcl}
&\b<0,&\a(p-1)-2\b\geq0\mbox{ and }2\a-\b(N-2)>0\\
\mbox{or}&\b\geq0,&\a(p-1)-2\b\geq0\mbox{ and }2\a-\b N>0.
\end{array}
\right.
\end{equation}
Let $\w\in(-1,1)$ and $\fw\in\hu$ be a ground state solution of
\begin{equation*}
-\D \fw +(1-\w^2)\fw- |\fw|^{p-1}\fw=0.
\end{equation*}
Then
$$
S(\fw)=\min\{S(v)\big|v\in\hu\setminus\{0\},\Kab(v)=0\}
$$
where
\begin{eqnarray*}
S(v)&:=&\frac{1}{2}\nldd{\nabla v}+\frac{1-\w^2}{2}\nldd{v}-\frac{1}{p+1}\nlpsps{v}. \\
\Kab(v)&:=&\textstyle\frac{2\a-\b(N-2)}{2}\nldd{\nabla v}+\frac{(2\a-\b N)(1-\w^2)}{2}\nldd{v}-\frac{\a(p+1)-\b N}{p+1}\nlpsps{v}.
\end{eqnarray*}
\end{thm}

The functional $\Kab$ is based on the rescaling $\vl(\,\cdot\,):=\la v(\lb\,\cdot\,)$ for $v\in\hu$, precisely,
$
\Kab(v)=\frac{\partial}{\partial \l}S(\vl)_{|\l=1}.
$
The main idea of the proof of Theorem \ref{thm:snlkgp} is to use rescaled functions to construct for any $v\in\hu$ such that $\Kab(v)=0$ a path in $\G$ attaining his maximum at $v$.

It is also of interest to consider a limit case of Theorem \ref{thm:snlkgp}.
\begin{thm}\label{thm:limitcase}
Let $\a,\b\in\R$ be such that
\begin{equation}\label{eq:limitcaseofab}
\left\{
\begin{array}{rcl}
&\b<0,&\a(p-1)-2\b\geq0\mbox{ and } 2\a-\b(N-2)=0\\
\mbox{or}&\b>0,&\a(p-1)-2\b\geq0\mbox{ and } 2\a-\b N=0.
\end{array}
\right.
\end{equation}
Let $\w\in(-1,1)$ and $\fw$ be a ground state solution of
\begin{equation*}
-\D \fw +(1-\w^2)\fw- |\fw|^{p-1}\fw=0.
\end{equation*}
Then
$$
S(\fw)=\min\{S(v)\big|v\in\hu\setminus\{0\},\Kab(v)=0\}.
$$
\end{thm}
\begin{rmq}
Looking to the proofs of Theorems \ref{thm:snlkgp} and \ref{thm:limitcase} one
see that our Theorems remain unchanged when $(1- \omega^2)$ is replaced by any
$m>0$. We choose however to present our results in the setting of
\cite{lot,ot1,ot2}.
\end{rmq}
For $(\a,\b)=(\frac{N}{2},1) $, Theorem \ref{thm:limitcase} gives a simpler
proof of a variational characterization of the ground state proved by
Berestycki and Cazenave \cite{bc} for $1+\frac{4}{N}<p<1+\frac{4}{N-2}$ and by
Nawa \cite[Proposition 2.5]{n} for $p=1+\frac{4}{N}.$ This characterization is
at the heart of the classical result of Berestycki and Cazenave \cite{bc}
dealing with the instability of the ground states of nonlinear Schr\"{o}dinger
equations.

For our second direction of application we consider the instability of the
stationary solutions of
\begin{equation}\label{eq:nlkgg1}
\utt-\D u=g(u).
\end{equation}
In 1985, Shatah established in \cite{sh} that under the conditions (g0)-(g3)
the radial ground states solutions associated with the standing waves
corresponding to $\w=0$ are unstable when $N\geq3$. Under stronger hypothesis,
but in any dimension and for non necessary radial solutions, Berestycki and
Cazenave \cite{bc} had previously proved that these ground states are unstable
by blow up in finite time. In \cite{sh}, instability may occur by blow up in
infinite time, in the sense that the $\hu$-norm of a solution starting close to
a ground state goes to infinity when $t\to+\infty$. Here, we show that the same
result hold when $N=2$.

We make the following hypothesis on the existence and properties of solutions for (\ref{eq:nlkgg1}).

\noindent\textbf{Assumption H.}
\textit{
For all $(\uo,\vo)\in\huradrd\times\ldradrd$ there exist $0<T\leq+\infty$ and $u:[0,T)\times\Rd\to\C$ such that
\begin{itemize}
\item $(u(0),u_t(0)) = (\uo, \vo)$,
\item $u$ (resp. $\ut$) is weakly continuous in $\huradrd$ (resp. $\ldradrd$),
\item $u$ satisfies (\ref{eq:nlkgg1}) in the sense of distributions,
\item $E(u(t),u_t(t))\leq E(\uo,\vo)$ for all $t\in[0,T)$ (\emph{energy inequality}),
\item if $T<+\infty$, there exists $(t_n)\subset[0,T)$ such that $t_n\to T$ as $n\to+\infty$ and $\lim_{t_n\to T}\nhurd{u(t_n)}=+\infty$ (\emph{blow-up alternative}),
\end{itemize}
}
The \emph{energy functional} $E$ is defined for $u\in\hu$ and $v\in\ld$ by
$$
E(u,v):=\frac{1}{2}\nldd{v}+\frac{1}{2}\nldd{\nabla u}-\intrd G(u)dx.
$$
In what follows, as above, we write $\hurad$ (resp. $\ldrad$) for the space of radial
functions of $\hu$ (resp. $\ld$).

\begin{rmq}
When $N\geq 3$, Shatah claims that Assumption H holds under (g0)-(g3) without
any additional restrictions. For others dimensions, Assumption H is known to
hold under stronger assumptions on $g$, see, for example, \cite[Chapter 6]{ch}.
From now on a solution of (\ref{eq:nlkgg1}) with initial data $(\uo,\vo)$ will
refer to a solution of (\ref{eq:nlkgg1}) with initial data $(\uo,\vo)$ as given
by Assumption H.
\end{rmq}

Our third main result is the following
\begin{thm}\label{thm:instability}
Assume $N=2$, (g0)-(g3) and Assumption H. Let $\f$ be a radial ground state of
(\ref{eq:nlsf1}). Then $\f$ viewed as a stationary solution of
(\ref{eq:nlkgg1}) is strongly unstable. Namely for all $\e>0$ there exist
$\ue\in\hurd$, $T_\e\in(0,+\infty]$ and $(t_n)\subset(0,T_\e)$ such that
$\nhurd{\f-\ue}<\e$ and $\lim_{t_n\to T_\e}\nhurd{u(t_n)}=+\infty$, where
$u(t)$ is a solution of (\ref{eq:nlkgg1}) with initial data $(\ue,0)$.
\end{thm}

It is still an open question to describe what happen in dimension $N=1$.
Indeed, the use of the radial compactness lemma of Strauss (see Lemma
\ref{lem:strauss}) restricts our proof to dimensions $N\geq2$. A partial answer
is given by the work of Berestycki and Cazenave : for nonlinearities satisfying
some additional assumptions (see \cite[(H.3)]{bc}), the stationary solutions
are unstable. 

We do hope that the methods developed in this paper will find other areas of
applications. In that direction, we mention the work \cite{le1} in which the
variational characterization $c=m$ derived from \cite{jt1,jt2} is essential to
get an alternative, more general proof of the classical result of Berestycki
and Cazenave \cite{bc} on the instability by blow-up for nonlinear Schr\"{o}dinger equations.

This paper is organized as follows. In Section \ref{sec:varcar} we prove
Theorem \ref{thm:snlkgp} and Theorem \ref{thm:limitcase}. In Section
\ref{sec:instability} we prove Theorem \ref{thm:instability}. The proof that the results of \cite{jt1,jt2} extend to the complex case along with a technical lemma are given in the Appendix.

\noindent\textbf{Acknowledgments.} The authors wish to thank Masahito Ohta and
Grozdena Todorova for the interest they have taken in this work and for
fruitful discussions. They are also grateful to Mariana H\u{a}r\u{a}gu\c{s} for
fruitful discussions.

\section{Variational characterizations of the ground states}\label{sec:varcar}

In this section, we consider (\ref{eq:snlkgf}) with a power type nonlinearity :
\begin{equation}\label{eq:snlkgp}
-\D \fw +(1-\w^2)\fw- |\fw|^{p-1}\fw=0
\end{equation}
where $1<p<1+4/(N-2)$ and $|\w|<1$. For this nonlinearity it is known (see
\cite[Section 8.1]{c} and the references therein) that there exists a unique
positive radial ground state $\fw\in H^1(\R^N, \R)$ of (\ref{eq:snlkgp}) and
that all ground states are of the form $e^{i \theta}\fw(\cdot -y)$ for some
fixed $\theta \in \R$ and $y \in \R^N$.  The standing waves $e^{i\w t}\fw$ are
solutions of the nonlinear Klein-Gordon equation
\begin{equation}\label{nlkgp}
\utt-\D u+u=|u|^{p-1}u
\end{equation}
and the natural functional associated with (\ref{eq:snlkgp}) becomes
$$
S(v):=\frac{1}{2}\nldd{\nabla v}+\frac{1-\w^2}{2}\nldd{v}-\frac{1}{p+1}\nlpsps{v}.
$$

Various results of instability for the standing waves of (\ref{nlkgp}) were
recently proved in \cite{lot,ot1,ot2}. For instance, it was proved in
\cite{ot1} that for any $1<p<1+4/(N-2)$ the standing wave associated with a
ground state of (\ref{eq:snlkgp}) is strongly unstable by blow up if $\w^2\leq
(p-1)/(p-3)$ and $N\geq3$. In \cite{ot2}, a result of strong instability was
showed for the optimal range of parameter $\w$  in dimension $N\geq2$ (namely
$|\w|<\w_c$, where $\w_c$ was determined in \cite{ss}). In both cases, it is
central in the proofs that the ground states can be characterized as minimizers
on constraints having all the form
$$
\scrKab:=\{ v\in\hu\setminus\{0\}\big|\Kab(v)=0 \}
$$
for some $\a,\b\in\R$. Recall that the functional $\Kab$ is defined for $v\in\hu$ by
\begin{eqnarray*}
\Kab(v)&:=&\textstyle\frac{\partial}{\partial \l}S(\la v(\lb \,\cdot\,))_{|\l=1}\\
&=&\textstyle\frac{2\a-\b(N-2)}{2}\nldd{\nabla v}+\frac{(2\a-\b N)(1-\w^2)}{2}\nldd{v}-\frac{\a(p+1)-\b N}{p+1}\nlpsps{v}.
\end{eqnarray*}
For example, it is proved in \cite{ot1} that the ground states are minimizer of $S$ on $\scrKab$ for $(\a,\b)=(1,0)$ and $(\a,\b)=(0,-1/N)$ (see \cite[(2.1)]{ot1}) whereas in \cite{ot2}, the values of $(\a,\b)$ considered are $(\a,\b)=(N/2,1)$ if $p\geq1+4/N$ (see \cite[(2.11)]{ot2}) and $(\a,\b)=(2/(p-1),1)$ if $p<1+4/N$ (see \cite[(2.18)]{ot2}). Recently, Liu, Ohta and Todorova \cite{lot} extended the approach of \cite{ot1} to the dimensions $N=1,2$. Once more, a main feature of their proof is to minimize $S$ on $\scrKab$, but this time with
$$
\a=\frac{(p-1)-(p+3)\w^2}{2(p-1)\w^2},\,\b=-1.
$$

In \cite{lot,ot1,ot2}, the proofs that the ground states are minimizers of $S$ on $\scrKab$ follow similar schemes. First, one has to show the convergence of a minimizing sequence to some function solving a Lagrange equation. After that, the difficulty is to get rid of the Lagrange multiplier. For each choice of $(\a,\b)$, long computations are involved to prove that the Lagrange multiplier is $0$ and to conclude that the obtained function is in fact a solution of (\ref{eq:snlkgp}).

Our proof of Theorem \ref{thm:snlkgp} relies on the following lemma. We recall
that $\G$ is defined in (\ref{Gamma}).
\begin{lem}\label{lem:path}
Let $\a,\b\in\R$ satisfy (\ref{eq:rangeofab}). Then for all $v\in\scrKab$ we can construct a path $\g$ in $\G$ such that
$$
\max_{t\in[0,1]}S(\g(t))=S(v).
$$
\end{lem}

\begin{proof}
Let $v\in\scrKab$. For all $\l\in(0,+\infty)$ we define $\vl\in\hu$ by $\vl(\,\cdot\,):=\la v(\lb \,\cdot\,)$. The idea is to construct the path such that $\g(t)=v_{Ct}$ for some $C>0$.

The first thing to check is that we can extend $\g$ at $0$ by continuity. Namely, we must show that under (\ref{eq:rangeofab}) we have $\lim_{\l\to0}\nhu{\vl}=0$. This is immediate if we remark that
$$
\nhud{\vl}=\l^{2\a-\b (N-2)}\nldd{\nabla v}+\l^{2\a-\b N}\nldd{v},
$$
and that (\ref{eq:rangeofab}) implies
$$
2\a-\b (N-2)>0\mbox{ and }2\a-\b N>0.
$$

The next step is to prove that $\l\to S(\vl)$ increases for $\l\in(0,1)$,
attains its maximum at $\l=1$ and decreases toward $-\infty$ on $(1,+\infty)$.
We have
$$
S(\vl)=\frac{\l^{2\a-\b (N-2)}}{2}\nldd{\nabla v}+\frac{(1-\w^2)\l^{2\a-\b N}}{2}\nldd{v}-\frac{\l^{(p+1)\a-\b N}}{p+1}\nlpsps{v}
$$
and from easy computations it comes
\begin{eqnarray*}
\l^{-(2\a-\b N-1)}\frac{\partial}{\partial \l}S(\vl)&=&\l^{2\b}\frac{2\a-\b(N-2)}{2}\nldd{\nabla v}+\frac{(2\a-\b N)(1-\w^2)}{2}\nldd{v}\\
&&-\l^{\a(p-1)}\frac{\a(p+1)-\b N}{p+1}\nlpsps{v}.
\end{eqnarray*}
Therefore, if $\a$ and $\b$ satisfy
\begin{equation}\label{eq:condab}
\left\{
\begin{array}{rcl}
&\b\neq0\mbox{ and}&\a(p-1)\geq 2\b \\
\mbox{or}&\b=0\mbox{ and}&\a(p-1)>0
\end{array}
\right.
\end{equation}
then
$$\left\{
\begin{array}{rcl}
&&\frac{\partial}{\partial \l}S(\vl)>0\mbox{ for }\l\in(0,1) ,\\
&&\frac{\partial}{\partial \l}S(\vl)<0\mbox{ for }\l\in(1,+\infty),\\
&&\lim_{\l\to+\infty}S(\vl)=-\infty.
\end{array}
\right.
$$
Since $\a>0$ when $\b=0$ in (\ref{eq:rangeofab}) it is clear that
(\ref{eq:condab}) hold under (\ref{eq:rangeofab}).

Finally, choosing $C$ large enough to have $S(v_C)<0$ and defining $\g:[0,1]\mapsto\hu$ by
$$
\g(0):=0\mbox{ and }\g(t):=v_{tC}
$$
we have a path satisfying the conclusion of the lemma.
\end{proof}

\begin{proof}[Proof of Theorem \ref{thm:snlkgp}]
Let $\fw$ be a least energy solution of (\ref{eq:snlkgp}) for $|\w|<1$. From
Lemma \ref{complex}  we know that
$$
c=m
$$
where $m$ is the least energy level and $c$ the mountain pass value (see (\ref{eq:m}) and (\ref{eq:c}) for the definitions of $m$ and $c$). Since $\fw$ is a solution of (\ref{eq:snlkgp}), $\fw\in\mathcal{C}^1$ and $\fw,$ $\nabla \fw$ are exponentially decaying at infinity (see, for example, \cite[Theorem 8.1.1]{c}); in particular, $x.\nabla\fw\in\hu$, and
$$
\Kab(\fw)=\frac{\partial}{\partial \l}S(\la \fw(\lb \,\cdot\,))\big|_{\l=1}=\dual{S'(\fw)}{\a\fw+\b x.\nabla\fw}=0.
$$
Thus $\fw\in\scrKab$ and
\begin{equation}\label{eq:thm:snlskg:1}
\min\{S(v)\big|v\in\scrKab\}\leq S(\fw)=c.
\end{equation}
Conversely, it follows from Lemma \ref{lem:path} that
\begin{equation}\label{eq:thm:snlskg:2}
c\leq \min\{S(v)\big|v\in\scrKab\}.
\end{equation}
To combine (\ref{eq:thm:snlskg:1}) and (\ref{eq:thm:snlskg:2}) finishes the proof.
\end{proof}

We now turn to the proof of Theorem \ref{thm:limitcase}. It follows the same
lines as for Theorem \ref{thm:snlkgp} : find a path reaching its maximum on the
constraint $\scrKab$ and use the equality $c=m$. The main difference is in the
way we construct the path : we still want to use the rescaled functions $\vl$,
but their $\hu-$norm does not any more converge to $0$ as $\l\to 0$. This
difficulty is overcome by gluing to $\{\vl\}_{\l>\lo}$ a path linking $0$ to
$\vlo$ for $\lo$ suitably chosen. The lemma is
\begin{lem}\label{lem:pathlimitcase}
Let $\a,\b\in\R$ satisfy (\ref{eq:limitcaseofab}). Then for all $v\in\scrKab$ we can construct a path $\g$ in $\G$ such that
$$
\max_{t\in[0,1]}S(\g(t))=S(v).
$$
\end{lem}

\begin{proof}
Let $v\in\scrKab$ and $\vlo(\cdot):=\lo^\a v(\lo^\b \,\cdot\,)$ for some
$\lo\in(0,1)$ whose value will be fixed later. Let $C>0$ be such that
$S(v_C)<0$ and consider the curves
\begin{eqnarray*}
\L_1&:=&\{\vl\big|\l\in[\lo,C]\}, \\
\L_2&:=&\{t\vlo\big|t\in[0,1]\}.
\end{eqnarray*}
To get a path as desired, we will glue the two curves $\L_1$ and $\L_2$. It is clear that as in the proof of Lemma \ref{lem:path}, $S$ attained its maximum on $\Lambda_1$ at $v$. Thus the only thing we have to check is that $t\mapsto S(t\vlo)$ is increasing on $[0,1]$.

We have
$$
\frac{\partial}{\partial t}S(t\vlo)=t(\nldd{\nabla \vlo}+ (1- \omega^2)
\nldd{\vlo}-t^{p-1}\nlpsps{\vlo}).
$$
If $\b>0$ and $\a=\b N/2$ (see (\ref{eq:limitcaseofab})), then $ \lo \to
\nld{\vlo}$ is constant. If  $\b<0$ and $\a=\b(N-2)/2$  then $\lo \to
\nld{\nabla \vlo}$ is constant. Moreover, we have in any case
$$
\lim_{\lo\to0}\nlpsps{\vlo}=0.
$$
Therefore, if $\lo\in(0,1)$ is small enough we have
$$
\frac{\partial}{\partial t}S(t\vlo)>0\mbox{ for }t\in(0,1).
$$
To define $\g:[0,1]\mapsto\hu$ by
$$
\left\{
\begin{array}{rcl}
\g(t)&=&\frac{Ct}{\lo}\vlo\mbox{ for }t\in[0,\frac{\lo}{C})\\
\g(t)&=&v_{Ct}\mbox{ for }t\in[\frac{\lo}{C},1]\\
\end{array}
\right.
$$
gives us the desired path.
\end{proof}

\begin{proof}[Proof of Theorem \ref{thm:limitcase}]
The proof is identical to the proof of Theorem \ref{thm:snlkgp} with Lemma \ref{lem:path} replaced by Lemma \ref{lem:pathlimitcase}.
\end{proof}

\section{Orbital instability for a generalized nonlinear Klein-Gordon equation}\label{sec:instability}

In this section, we consider the nonlinear Klein-Gordon equation with a general nonlinearity
\begin{equation}\label{eq:nlkgg2}
\utt-\D u=g(u).
\end{equation}
In \cite{sh}, Shatah proved that for $N\geq3$, under (g0)-(g3), the radial
ground
 states solutions of
\begin{equation}\label{eq:nlsf2}
-\D \f = g(\f), \quad u \in \hu
\end{equation}
viewed as stationary solutions of (\ref{eq:nlkgg2}) are unstable in the sense of Theorem \ref{thm:instability}.

The restriction to $N\geq3$ has its origin in, at least, two reasons.

First, one needs to control the decay in $|x|$ of $u(t,x)$ uniformly in $t$.
This appears in the proofs of Proposition \ref{prop:blow-up} and Lemma
\ref{lem:bgk}. For this control, the following compactness lemma due to Strauss
\cite{st} is used.
\begin{lem}\label{lem:strauss}
Let $N\geq2$ and $v\in\hurad$. Then
$$
|v(x)|\leq C|x|^{\frac{1-N}{2}}\nhu{u}\mbox{ a.e.}
$$
with $C$ independent of $x$ and $u$. In particular, the following injection is compact
$$
\hurad\hookrightarrow \lq \mbox{ for }2< q<2^\star,
$$
where $2^\star=\frac{2N}{N-2}$ if $N\geq 3$ and $2^\star=+\infty$ if $N=2$.
\end{lem}
Actually, to use this lemma only $N\geq 2$ is necessary.

A second reason for the restriction $N\geq3$ in \cite{sh} is found in the use
of a constraint based on Pohozaev's identity to derive a variational
characterization of the ground states, to define an invariant set, and, most
important, to choose suitable initial data close to the ground states. Thanks
to our approach, we arrive on this second point to require only $N\geq 2$.

Our proof will make use of the following variational characterization of the
ground states.
\begin{lem}\label{lem:minp}
Let $\f\in\hurd$ be a ground state of (\ref{eq:nlsf2}). Then
\begin{equation}\label{eq:minp}
S(\f)=m=\min_{v\in\scrP}S(v)
\end{equation}
where
$$
\scrP:=\{ v\in\hurd\setminus\{0\}\big|P(v)=0 \}
$$
with $\displaystyle P(v):=\intrd G(v)dx$ for $v\in\hu$.
\end{lem}
This lemma was proved in \cite{bgk} when $v \in H^1(\R^N, \R)$. It can trivially be extended to $v \in \hu$, see
\cite{cjs}.

\begin{rmq}
The functional $P$ is related to the so-called Pohozaev identity (see
\cite[Proposition 1]{bl}, \cite{st}): for  $N\geq 1$, any solution $v\in\hu$ of
(\ref{eq:nlsf2}) satisfies
$$
\frac{N-2}{2}\nldd{\nabla v}-N\intrn G(v)dx=0.
$$
A main feature of the dimension $N=2$ is that we lose the control on the $\ld-$norm of $\nabla v$.
\end{rmq}
\begin{rmq} For $N\geq 3$, Shatah also showed that the radial ground states are minimizers of $S$ among
all non trivial functions satisfying Pohozaev identity (see \cite[Proposition
1.5]{sh}). His method consists in proving that the minimization problem has a
solution and then to eliminate the Lagrange multiplier. In fact, as it is done in \cite[Lemma 3.1]{jt1}, a shorter proof can be performed by simply establishing a correspondence with a minimization problem already solved in \cite{bl}. 
\end{rmq}

The scheme of the proof is the following : first, define a set
$\I\subset\huradrd\times\ldradrd$ such that any solution of (\ref{eq:nlkgg2})
with initial data in $\I$ stays in $\I$ for all time and blows up, then prove
that the ground states can be approximated by functions in $\I$.

Let $\I$ be defined by
$$
\I:=\{ u\in\huradrd\setminus\{0\},v\in\ldradrd\big|E(u,v)<m, P(u)>0 \}.
$$
We begin by proving an equivalence between two variational problems.
\begin{lem}\label{lem:minpeq}
We have
$$
m=\min_{v\in\scrP}S(v)=\min\{ T(v)\big|v\in\hurd\setminus\{0\}, P(v)\geq 0 \},
$$
where $\displaystyle T(v):=\frac{1}{2}\nldd{\nabla v}$.
\end{lem}
\begin{proof}
Let $v\in\hurd$. If $v\in\scrP$, then $v$ satisfies $T(v)=S(v)$ and thanks to Lemma \ref{lem:minp}, $T(v)\geq m$. Suppose that $P(v)>0$. For $\l>0,$ define $\vl(\,\cdot\,):=\l v(\l\,\cdot\,)$. We claim that there exists $\lo<1$ such that $P(\vlo)=0$. Indeed, by (g1)-(g2), for all $\a>0$ there exists $\Ca>0$ such that for $s>0$
$$
g(s)\leq \frac{-\r s}{2}+2s\a\Ca e^{\a s^2}.
$$
We recall that $\r>0$ is given in (g1) by $\lim_{s\to 0} g(s)s^{-1}=-\r$. Therefore, for $s>0$ we have
$$
G(s)\leq \frac{-\r s^2}{4}+\Ca (e^{\a s^2}-1)
$$
and
\begin{equation}\label{eq:lem:var1}
\intrd G(\vl)\leq \frac{-\r \nldd{\vl}}{4}+\Ca\intrd (e^{\a \vl^2}-1)dx.
\end{equation}
We remark that $\nldd{\vl}=\nldd{v}$ and
$$
\intrd (e^{\a \vl^2}-1)dx=\l^{-2}\intrd (e^{\a \l^2 v^2}-1)dx.
$$
For $\l<1$ we have
$$
\l^{-2}(e^{\a \l^2 v^2(x)}-1)<e^{\a v^2(x)}-1\mbox{ for all }x\in\Rd,
$$
and by Moser-Trudinger inequality (see \cite[Theorem 8.25]{a}) there exists $\a>0$ such that $(e^{\a v^2}-1)\in L^1(\Rd)$. Hence, Lebesgue's Theorem gives
$$
\intrd (e^{\a \vl^2}-1)dx\to 0\mbox{ when }\l\to 0.
$$
Coming back to (\ref{eq:lem:var1}) this means that
$$
\intrd G(\vl)<0\mbox{ for }\l>0\mbox{ small enough},
$$
and by continuity of $P$ this proves the claim.

Now, we have
$$
\inf_{u\in\scrP}S(u)\leq S(\vlo)=T(\vlo)=\lo^2T(v)<T(v),
$$
and the lemma is proved.
\end{proof}

Next we prove that the set $\I$ is invariant under the flow of (\ref{eq:nlkgg2}).

\begin{lem}\label{lem:flowinvariant}
Let $(\uo,\vo)\in\I$, $0<T\leq+\infty$ and $u(t)$ a solution of
(\ref{eq:nlkgg2}) on $[0,T)$ with initial data $(\uo,\vo)$. Then
$(u(t),\ut(t))\in\I$ for all $t\in[0,T)$.
\end{lem}

\begin{proof}
Let
$$
\tz:=\inf\big\lgroup \{ t\in[0,T)\big|P(u(t))\leq0
\}\cup\{+\infty\}\big\rgroup.
$$
Assume by contradiction that $\tz\neq+\infty$ and consider
$(t_n)\subset(\tz,T)$ such that $t_n\downarrow\tz$ with $P(u(t_n)) \leq 0$. By
Assumption H, $u(t_n)\goestoweak u(\tz)$ weakly in $\hurd$. Thus we have
\begin{equation}\label{eq:lemminvar1}
T(u(\tz))\leq\liminf_{n\to +\infty} T(u(t_n))\leq\liminf_{n\to +\infty}
\left[T(u(t_n))-P(u(t_n))\right].
\end{equation} Moreover
\begin{equation}\label{eq:lemminvar2}
\liminf_{n\to +\infty} \left[T(u(t_n))-P(u(t_n))\right]=\liminf_{n\to +\infty}
S(u(t_n))\leq \liminf_{n\to+\infty}E(u(t_n),u_t(t_n))
\end{equation}
and by the energy inequality in Assumption H we get
\begin{equation}\label{eq:lemminvar3}
\liminf_{n\to+\infty}E(u(t_n),u_t(t_n))\leq E(\uo,\vo).
\end{equation}
Recalling that $(\uo,\vo)\in\I$, we have
\begin{equation}\label{eq:lemminvar4}
E(\uo,\vo)<m.
\end{equation}
Combining (\ref{eq:lemminvar1})-(\ref{eq:lemminvar4}) gives
\begin{equation}\label{eq:lemminvar5}
T(u(\tz))<m.
\end{equation}

Now, take $(\widetilde{t_n})\subset(0,\tz)$  such that
$\widetilde{t_n}\uparrow\tz$. By Lemma \ref{lem:uppersemicontinuity}, $v\to
P(v)$ is upper weakly semi-continuous, thus
\begin{equation}\label{eq:lemminvar6}
P(u(\tz))\geq\limsup_{n\to+\infty}P(u(\widetilde{t_n}))\geq 0.
\end{equation}
Now together  (\ref{eq:lemminvar5}) and (\ref{eq:lemminvar6}) lead to a
contradiction with Lemma \ref{lem:minpeq}.
\end{proof}

The following lemma is a key step in the proof.

\begin{lem}\label{lem:key}
Let $(\uo,\vo)\in\I$ and $u(t)$ an associated solution of (\ref{eq:nlkgg2}) in
$[0,T)$. Then there exists $\d>0$ such that $P(u(t))>\d$ for all $t\in[0,T)$.
\end{lem}

\begin{proof}
Indeed, assume by contradiction that there exists a sequence $(t_n)$ such that $P(u(t_n))\to 0$ as $n\to+\infty$. Then
\begin{eqnarray*}
T(u(t_n))&=&S(u(t_n))+P(u(t_n))\\
&\leq& E(u(t_n),\ut(t_n))+P(u(t_n)).
\end{eqnarray*}
By the energy inequality in Assumption H this implies
\begin{equation*}
T(u(t_n))\leq E(\uo,\vo)+P(u(t_n))
\end{equation*}
and thus
\begin{equation}\label{eq:33}
T(u(t_n))< m+P(u(t_n))-\n
\end{equation}
with $\n:=m-E(\uo,\vo)>0$ since $(\uo,\vo)\in\I$.
For $n$ large enough we have
$$
0\leq P(u(t_n))<\n/2
$$
and thus (\ref{eq:33}) gives
$$
T(u(t_n))<m-\frac{\n}{2},
$$
which contradicts the result of Lemma \ref{lem:minpeq}.
\end{proof}

The proof of Theorem \ref{thm:instability} relies on the following proposition.

\begin{prop}\label{prop:blow-up}
Let $(\uo,\vo)\in\I$ and $u(t)$ an associated solution of (\ref{eq:nlkgg2}) on
$[0,T)$. Then there exists $(t_n)\subset(0,T)$ such that $\lim_{t_n\to
T}\nhurd{u(t_n)}=+\infty$,
\end{prop}

\begin{proof}
The proof of Proposition \ref{prop:blow-up} is similar to the proof of Theorem 2.3 in \cite{sh}, thus we just indicate the main steps. First, if $T<+\infty$, the assertion of Proposition \ref{prop:blow-up} is just the blow up alternative in Assumption H. Thus we suppose $T=+\infty$. Following the line of the proof of Theorem 2.3 in \cite{sh}, it is not hard to see that there exists $0<\eta<\d$ (where $\d$ is given by Lemma \ref{lem:key}) such that
\begin{equation}\label{eq:propblowup1}
2\intrd G(u)\,dx-\eta\leq-\frac{\partial}{\partial t}\Re\intrd \t(t,x)\utb x.\nabla u dx
\end{equation}
where $\t:[0,+\infty)\times\Rd\mapsto\R$ is such that
\begin{equation}\label{eq:propblowup2}
|\t(t,x)|\leq Ct/\ln(t)
\end{equation}
for all $(t,x)\in[0,+\infty)\times\Rd$. To combine (\ref{eq:propblowup1}) and Lemma \ref{lem:key} gives
\begin{equation}\label{eq:propblowup3}
\d\leq-\frac{\partial}{\partial t}\Re\intrd \t(t,x)\utb x.\nabla u dx.
\end{equation}
Hence, by integrating (\ref{eq:propblowup3}) we find
\begin{equation}\label{eq:propblowup4}
\d t\leq -\Re\intrd \t(t,x)\utb x.\nabla u dx +\Re\intrd \t(0,x)\overline{\vo}\ x.\nabla \uo dx.
\end{equation}
Now, by (\ref{eq:propblowup2}) and (\ref{eq:propblowup4}) there exists $C>0$ such that
\begin{equation}\label{eq:propblowup5}
\ln(t)\d\leq C(1+\nld{\nabla u(t)}\nld{\ut(t)}).
\end{equation}
But, thanks to the energy inequality $\nld{\ut(t)}$ is bounded, and $\nld{\nabla u(t)}$ is bounded by assumption, therefore, for $t$ large enough we reach a contradiction in (\ref{eq:propblowup5}).
\end{proof}

In dimension $N\geq 3$, it is easily seen that for $\l<1$ the dilatation of a
ground state $\fl(\,\cdot\,):=\f(\frac{\,\cdot\,}{\l})$ gives a sequence of
initial data in $\I$ converging to this ground state. This property, combined
with the equivalent of Proposition \ref{prop:blow-up}, gives immediately the
instability of the ground states in \cite{sh}. This is not the case any more in
dimension $N=2$ where the dilatation $\fl(\,\cdot\,):=\f(\frac{\,\cdot\,}{\l})$
leaves $P$ and $T$ invariant. To overcome this difficulty, we borrow and adapt
an idea of \cite[Proposition 2]{bjt} which consists in using separately (and
successively) a dilatation and a rescaling to get initial data in $\I$ close to
the ground states.

\begin{lem}\label{lem:initialdata}
Let $\f \in H^1(\R^2)$ be a ground state of (\ref{eq:nlsf2}). For all $\e>0$
there exists $\fe$ such that
$$
\nhurd{\f- \fe}<\e,\;S(\fe)<S(\f),\;P(\fe)>0.
$$
\end{lem}

\begin{proof}
For $\l,\m>0$ consider $\flm(\,\cdot\,):=\l\f(\frac{\,\cdot\,}{\m})$. Then
$$
\frac{\partial}{\partial \l}S(\flm)=\l^2\nldd{\nabla \f}-\m^2\intrd g(\l\f)\overline{\f}dx.
$$
To multiply (\ref{eq:nlsf2}) by $\overline{\f}$ and integrate gives us
$$
\nldd{\nabla\f}=\intrd g(\f)\overline{\f}dx.
$$
Hence, for $\l=1$ we get
$$
\frac{\partial}{\partial \l}S(\flm)\big|_{\l=1}=(1-\m^2)\nldd{\nabla \f}.
$$
Thus, for all $\m>1$, there exists $\lm>0$ such that
$$
\frac{\partial}{\partial \l}S(\flm)<0\mbox{ for }\l\in(1-\lm,1+\lm)
$$
and therefore
\begin{equation}\label{eq:initialdata1}
S(\flm)<S(\f)\mbox{ for }\l\in(1,1+\lm).
\end{equation}
Now,
$$
\frac{\partial}{\partial \l}P(\flm)_{\l=1}=\m^2\intrd g(\f)\overline{\f}dx=\m^2\nldd{\nabla\f}>0.
$$
Thus, for all $\m>0$, there exists $\Lm$ such that
$$
\frac{\partial}{\partial \l}P(\flm)>0\mbox{ for }\l\in(1-\Lm,1+\Lm)
$$
and therefore
\begin{equation}\label{eq:initialdata2}
P(\flm)>0\mbox{ for }\l\in(1,1+\Lm).
\end{equation}
Finally, from (\ref{eq:initialdata1})-(\ref{eq:initialdata2}), for $\l,\m>1$ close enough to $1$
we get the desired result.
\end{proof}

\begin{proof}[Proof of Theorem \ref{thm:instability}]
Let $\e>0$ and $\fe$ given in Lemma \ref{lem:initialdata}. Then $(\fe,0)$ satisfies
$$
E(\fe,0)=S(\fe)<m\mbox{ and }P(\fe)>0,
$$
namely $(\fe,0)\in\I$. Theorem \ref{thm:instability} follows now from Proposition \ref{prop:blow-up}.
\end{proof}

\section{Appendix}

\begin{lem}\label{complex}
Let $m$ denote the least energy level defined in (\ref{eq:m}) and $c$ the
mountain pass level defined in (\ref{eq:c}). Then $m=c$.
\end{lem}

\begin{proof}
In \cite[Theorem 0.2]{jt1} for $N \geq 2$ and  \cite[Theorem 1.2]{jt2} for
$N=1$ it is shown that when the class $\G$ is replaced by
$$
\tilde{\G}:=\{ \g\in\mathcal{C}([0,1],H^1(\R^N, \R)),\g(0)=0,\,S(\g(1))<0 \}
$$
one has
$$
\tilde{c}:=\inf_{\g\in\G}\max_{t\in[0,1]}S(\g(t)) = \tilde{m}
$$
where $\tilde{m}$ is the least energy level among real valued solutions of
(\ref{eq:nlsf1}). From \cite{bgk,bl,cjs} we know that $\tilde{m} = m$. Also
trivially $c \leq \tilde{c}$. Now for each $\gamma \in \G$ we observe that
setting $\tilde{\gamma}(t) = |\gamma(t)|$ one has
$$ ||\nabla \tilde{\gamma}(t)||_2^2 \leq || \nabla \gamma(t)||_2^2 \quad  \mbox{
and } \int_{\R^N}G(\tilde{\gamma}(t))dx =  \int_{\R^N}G(\gamma(t))dx. $$ Thus
$\tilde{\gamma} \in \tilde{\G}$ and $S(\tilde{\gamma}) \leq S(\gamma). $ This
show that $\tilde{c} \leq c$ and ends the proof.
\end{proof}

Now we prove the upper weakly semicontinuity of $P$. We begin by a convergence
lemma
\begin{lem}\label{lem:bgk}
Let $H\in\mathcal{C}(\R,\R)$ be such that
\begin{itemize}
\item[(H1)] For all $\a>0$ there exists $\Ca>0$ such that $|H(s)|\leq\Ca (e^{\a s^2}-1)$ for all $s\geq1$,
\item[(H2)] $H(s)=o(s^2)$ when $s\to 0$.
\end{itemize}
Let $(u_n)\subset\huradrd$ be a sequence bounded in $\hurd$ such that $u_n\to u$ a.e. Then we have
$$
H(u_n)\to H(u)\mbox{ in }L^1(\Rd).
$$
\end{lem}

This lemma was proved in \cite[Lemma 5.2]{bgk2}, the extended version of \cite{bgk}. We recall it here for the sake of completeness.

\begin{proof}[Proof of Lemma \ref{lem:bgk}]
From the continuity of $H$ we have $H(u_n)\to H(u)$ a.e. By a theorem of Vitali
(see, for example, \cite[p 380]{ni}), it is enough to prove
\begin{enumerate}[(i)]
\item for each $\e>0$ there exists $R>0$ such that $\displaystyle\int_{\Rd\setminus \{|x|<R\}}H(u_n)dx<\e$ for all $n\in\N$,
\item for each $\e>0$ there exists $\d>0$ such that $\displaystyle\int_{\{|x-y|<\d\}}H(u_n)dx<\e$ for all $y\in\{x\in\Rd\text{ such that }|x|<R\}$ (\emph{equiintegrability}).
\end{enumerate}

Let $\e>0$ be arbitrary chosen. From (H1)-(H2), for $\a>0$ there exists $\Ca>0$ such that for all $s\in\R$
$$
|H(s)|\leq \a s^2+\Ca (e^{s^2}-1).
$$
Thus, for any $R>0$
$$
\int_{\{|x|>R\}}|H(u_n)|\leq \a\nldd{u_n}+\Ca\int_{\{|x|>R\}} (e^{u_n^2}-1)dx.
$$
On one hand, since $(u_n)$ is bounded in $\ld$ we can take $\a>0$ small enough such that
$$
\a\nldd{u_n}<\frac{\e}{2}.
$$
On the other hand, from Lemma \ref{lem:strauss} there exists $C$ such that
$$
\Ca\int_{\{|x|>R\}} (e^{u_n^2}-1)dx\leq \Ca\int_{\{|x|>R\}} (e^{C|x|^{-1}}-1)dx
$$
and for $R>0$ chosen large enough we have
$$
\Ca\int_{\{|x|>R\}} (e^{C|x|^{-1}}-1)dx<\frac{\e}{2}.
$$
Therefore, (i) is satisfied.

For (ii), we first remark that, by (H1) and Moser-Trudinger inequality, there exists $\a>0$ and $M>0$ such that
$$
\int_{\{|x|<R\}}H(u_n)dx\leq \int_{\{|x|<R\}}e^{\a u_n^2}dx<M\text{ for all }n\in\N
$$
In particular, then $H(u_n)$ is bounded in $L^r(|x|<R)$ for any $1<r<+\infty$.  Hence (ii) holds by de La Vall\'{e}e Poussin equiintegrability lemma.
\end{proof}

\begin{lem}\label{lem:uppersemicontinuity}
The functional $P(v)=\intrn G(v)dx$ is of class $C^1$ and upper weakly
semi-continuous in $\hu$.
\end{lem}

\begin{proof}
It is standard to show that under (g2), $P \in \mathcal{C}^1(\hu,\R)$. Now let
$v_n\goestoweak v$ in $\hu$. Using (g1)-(g2), we can decompose $G$ in
$$
G(s)=-\r s^2+H(s)
$$
where $H$ satisfies the hypothesis of Lemma \ref{lem:bgk}. Hence
$$
\intrn H(v_n)dx\to\intrn H(v)dx \text{ when }n\to +\infty.
$$
Since $v\to -\nld{v}$ is upper weakly semicontinuous, this conclude the proof.
\end{proof}

\end{document}